\documentclass[12pt,a4paper,oneside]{amsart}
\usepackage{amsfonts, amsmath, amssymb, amsthm}
\usepackage[colorlinks=true,citecolor=blue]{hyperref}
\usepackage[margin=1.4in]{geometry}
\usepackage{txfonts}
\usepackage[T1]{fontenc}
\usepackage{bbm}
\usepackage{mathtools}
\usepackage{graphicx}
\usepackage{enumerate}
\usepackage{verbatim}
\usepackage{tikz}
\usetikzlibrary{calc}

\usepackage[colorinlistoftodos,prependcaption,textsize=small]{todonotes}

\begingroup
    \makeatletter
    \@for\theoremstyle:=definition,remark,plain\do{%
        \expandafter\g@addto@macro\csname th@\theoremstyle\endcsname{%
            \addtolength\thm@preskip\parskip
            }%
        }
\endgroup

\newtheorem{theorem}{Theorem}
\newtheorem*{theorem*}{Theorem}
\newtheorem{lemma}[theorem]{Lemma}

\newtheorem*{claim*}{Claim}
\newtheorem*{quest*}{Question}
\newtheorem{claim}[theorem]{Claim}

\theoremstyle{definition}

\newtheorem{conj}[theorem]{Conjecture}

\newtheorem*{remark*}{Remark}

\newtheorem*{example*}{Example}

\newtheorem{obs}{Observation}

\newcommand{\trunk}{\textrm{Trunk}}

\newcommand{\ceil}[1]{\left \lceil {#1} \right \rceil}
\newcommand{\floor}[1]{\left \lfloor {#1} \right \rfloor}

\newcommand{\tree}{\textrm{Tree}}
\newcommand{\size}[1]{\left| #1 \right|}

\newcommand{\R}{\mathbb R}

\makeatletter
\providecommand*{\cupdot}{%
  \mathbin{%
    \mathpalette\@cupdot{}%
  }%
}
\newcommand*{\@cupdot}[2]{%
  \ooalign{%
    $\m@th#1\cup$\cr
    \hidewidth$\m@th#1\cdot$\hidewidth
  }%
}
\makeatother

\begin{document}

\title{Strong Erd{\H o}s--Hajnal properties in chordal graphs}
\author{Minho~Cho}
\address{Department of Mathematical Sciences, KAIST, Daejeon, South Korea}
\email{minhocho.math@gmail.com}

\author{Andreas~F.~Holmsen}
\address{Department of Mathematical Sciences, KAIST, Daejeon, South Korea}
\email{andreash@kaist.edu}

\author{Jinha~Kim}
\address{Discrete Mathematics Group, IBS, Daejeon, South Korea}
\email{jinhakim@ibs.re.kr}

\author{Minki~Kim}
\address{Division of Liberal Arts and Sciences, GIST, Gwangju, South Korea}
\email{minkikim@gist.ac.kr}

\thanks{M.~Cho and A.~F.~Holmsen were  supported by the National Research Foundation of Korea (NRF) grants funded by the Ministry of Science and ICT (NRF-2020R1F1A1A01048490).
J.~Kim was supported by the Institute for Basic Science (IBS-R029-C1). M~Kim was supported by Basic Science Research Program through the National Research Foundation of Korea (NRF) funded by the Ministry of Education (NRF-2022R1F1A1063424) and by GIST Research Project grant funded by the GIST in 2023.}

\date{\today}

\begin{abstract}
A graph class $\mathcal{G}$ has the strong Erd{\H o}s--Hajnal property (SEH-property) if there is a constant $c=c(\mathcal{G}) > 0$ such that for every member $G$ of $\mathcal{G}$, either $G$ or its complement has $K_{m, m}$ as a subgraph where $m \geq \left\lfloor c|V(G)|\right\rfloor$. We prove that the class of chordal graphs satisfy SEH-property with constant $c = 2/9$.

On the other hand, a strengthening of SEH-property which we call the colorful Erd{\H o}s--Hajnal property was discussed in geometric settings by Alon et al.~(2005) and by Fox et al.~(2012). Inspired by their results, we show that for every pair $F_1, F_2$ of subtree families of the same size in a tree $T$ with $k$ leaves, there exists subfamilies $F'_1 \subseteq F_1$ and $F'_2 \subseteq F_2$ of size $\theta \left( \frac{\ln k}{k} \left| F_1 \right|\right)$ such that either every pair of representatives from distinct subfamilies intersect or every such pair do not intersect.
Our results are asymptotically optimal.
\end{abstract}

\maketitle


\section{Introduction}

\subsubsection*{Background.} A classical conjecture of Erd{\H o}s and Hajnal \cite{ErdosHajnal1989} asserts that if $G$ is a graph on $n$ vertices which does not contain some fixed graph $H$ as an induced subgraph, then $G$ contains a clique or an independent set on at least $\lfloor n^\delta \rfloor$ vertices where $\delta>0$ is a constant depending only on the graph $H$. 
In general we say that a graph class $\mathcal{G}$ has the {\em Erd{\H o}s--Hajnal property} if there exists a constant $\delta = \delta(\mathcal{G})$ such that every graph in $\mathcal{G}$ on $n>1$ vertices contains a clique or an independent set of size $n^\delta$. (Here we use the term {\em graph class} to mean a family of graphs that is closed under taking induced subgraphs.)
\cite{ChudnovskySurvey, ErdosHajnal1989, FoxPach2008}

Rather than asking for a large clique or independent set, one variation of the Erd{\H o}s--Hajnal problem, asks for a large {\em bi-clique} as a subgraph in $G$ or in the complement of $G$. Here a {\em bi-clique of size $2k$} is a complete bipartite graph whose  vertex classes each consists of $k$ vertices.

A graph class $\mathcal{G}$ is said to have the {\em strong Erd{\H o}s--Hajnal property} (SEH-property) if there exists a constant $\varepsilon = \varepsilon(\mathcal{G})>0$ such that every graph $G\in \mathcal{G}$ on $n$ vertices or its complement $\overline{G}$ contains bi-clique of size $2 \left\lfloor\varepsilon n\right\rfloor$.
It was shown in \cite{alon-semi2005} that if a graph class has the SEH-property, then it also has the Erd{\H o}s--Hajnal property.\footnote{The reader should be warned that the name ``strong Erd{\H o}s--Hajnal property'' appears in the literature in various contexts. Here we are using the terminology introduced in \cite{FoxPachToth2010}.} A number of graph classes arising from discrete geometry have been shown to have the SEH-property, most notably are the cases of semi-algebraic graphs \cite{alon-semi2005} and intersection graphs of convex sets in the plane \cite{FoxPachToth2010}.
The goal of this paper is to study the SEH-property and related properties for some specific graph classes. 

\medskip

\subsubsection*{Our results.} The most general and powerful results regarding the SEH-property, typically do not give particular good bounds on the constant $\varepsilon$ (nor do they aim to do so). One of the goals of this paper is to provide (asymptotically) optimal constants for the SEH-property with respect to the following graph classes:
\begin{itemize}
    \item {\em Interval graphs.} An interval graph is the intersection graph of a finite family of intervals on the real line. That is, each vertex can be represented by an interval and two vertices are adjacent if and only if the corresponding intervals intersect.  
    \item {\em Cographs.} A cograph (complement-reducible graph) is a graph that can be obtained from a single vertex by complementation and disjoint union. Equivalently, it is a graph which does not contain the path on four vertices as an induced subgraph. 
    \item {\em Chordal graphs.} A chordal graph is a graph in which every cycle on four or more vertices has a chord, that is, there are no induced cycle on four or more vertices. Equivalently, a chordal graph is the intersection graph of a finite family of subtrees of an ambient tree \cite{Gavril1974}. (This is called the subtree representation of the chordal graph.)
\end{itemize}

\begin{theorem}\label{thm:seh}
The following graph classes satisfy the strong Erd{\H o}s--Hajnal property.

\begin{enumerate}
    \item \label{seh:int} Interval graphs with constant $\varepsilon = 1/4$.
    \item \label{seh:cog} Cographs with constant $\varepsilon = 1/4$.
    \item \label{seh:cho} Chordal graphs with constant $\varepsilon = 2/9$.
\end{enumerate}
\end{theorem}

\medskip

We now turn our attention to a variation of the SEH-property. We say that a graph class $\mathcal{G}$ has the {\em colorful Erd{\H o}s--Hajnal property} (CEH-property) if there exists a constant $\varepsilon_c  = \varepsilon_c(\mathcal{G}) > 0$ such that for any graph $G\in \mathcal{G}$ on $n$ vertices and {\em for any} partition of the vertex set $V(G)$ into parts of size $\lceil n/2 \rceil$ and $\lfloor n/2 \rfloor$, $G$ or its complement $\overline{G}$ contains a bi-clique of size $2\left\lfloor \frac{\varepsilon_c n}{2}\right\rfloor$ whose vertex classes belong to different parts of the given partition of $V(G)$. In other words, we ask not only for a large bi-clique in $G$ or its complement $\overline{G}$, but for one that respects an arbitrarily preassigned equipartition of the vertex set of $G$.

It was shown in \cite{alon-semi2005} that semi-algebraic graphs satisfy the CEH-property, but this general and powerful result does not give particularly good bounds on the constant involved. Our next goal is to determine (asymptotically) optimal constants for the CEH-property with respect to the same graph classes as in Theorem \ref{thm:seh}. However, it will be evident that the class of chordal graphs {\em does not} satisfy the CEH-property, and therefore we consider a refinement of this class.

Recall that the {\em leafage} of a chordal graph $G$, denoted by $\ell(G)$, is the minimum number of leaves of the ambient tree in a subtree representation of $G$. For an integer $k\geq 2$ let $\mathcal{T}_k$ denote the family of chordal graphs whose leafage is at most $k$. That is, 
\[\mathcal{T}_k = \{G : G \text{ is chordal with } \ell(G)\leq k\}.\]
This gives us an infinite chain  $\mathcal{T}_2 \subset \mathcal{T}_3 \subset \cdots \subset \mathcal{T}_\infty$ where  $\mathcal{T}_2$ is the class of interval graphs and $\mathcal{T}_\infty$ is the class of chordal graphs. 

\begin{theorem}\label{thm:ceh} The following graph classes satisfy the colorful Erd{\H o}s--Hajnal property. 
\begin{enumerate}
    \item Interval graphs with constant $\varepsilon_c = 1/3$.
    \item Cographs with constant $\varepsilon_c = 1/4$.
    \item The class $\mathcal{T}_k$ with constant $\varepsilon_c = \frac{\ln k}{20k}$.
\end{enumerate}
\end{theorem}

\subsubsection*{Basic terminology and notations}
As usual, a graph $G$ is an ordered pair $G = (V, E)$ consisting of a finite vertex set $V = V(G)$ and an edge set $E = E(G) \subset \binom{V}{2}$. In particular, all graphs in this paper are simple, having no loops and no parallel edges. 
The \emph{complement graph} of a graph $G$ is the graph $\overline{G} = (V, \binom{V}{2} - E(G))$.
The disjoint union of two sets $A$ and $B$ is denoted by $A \cupdot B$, and the \emph{disjoint union} of two given graphs $G_1 = (V_1, E_1)$ and $G_2 = (V_2, E_2)$ is the graph $(V_1 \cupdot V_2, E_1 \cupdot E_2)$. With a slight abuse of notation we denote this by $G_1 \cupdot G_2$.

A complete bipartite graph is simply called a \emph{bi-clique}. A bi-clique $K_{m,n}$ is \emph{balanced} if $m = n$. We define the \emph{size} of bi-cliques only for balanced ones and the size of $K_{m, m}$ is $2m$. 

Given a family $F$ of nonempty sets, the \emph{intersection graph} of $F$ is a graph $G = (V,E)$ such that $V = F$ and two vertices $A$ and $B$ are adjacent in $G$ if and only if $A \cap B \neq \emptyset$.
Let $G$ be the intersection graph of a family $F$ of nonempty sets, and let $F_1$ and $F_2$ be two disjoint subfamilies of $F$.
We say $F_1$ and $F_2$ {\em correspond to a bi-clique in $G$} if every member of $F_1$ intersects every  member of $F_2$. Similarly, we say $F_1$ and $F_2$ {\em correspond to a bi-clique in $\overline{G}$} if every member of $F_1$ is disjoint from every member of $F_2$.

For a vertex $v\in V(G)$, the {\em neighborhood} of $v$, denoted by  $N(v)$, is the set of all vertices  adjacent to $v$. 
The \emph{closed neighborhood} of $v$ is $N[v] := N(v) \cup \{v\}$. The \emph{degree} of $v$, denoted by $\text{deg}(v)$, is the number of adjacent vertices to $v$, and $\Delta(G)$ denotes the maximum degree over all vertices in $G$.

For a tree $T$, a \emph{leaf} of $T$ is a vertex of degree $1$ in $T$. Given a pair of vertices $u, v \in V(T)$, then we denote by $P_T(u, v)$  the unique path in $T$ connecting $u$ and $v$. More generally, for a vertex set $U = \{u_1, u_2, \ldots, u_n\} \subseteq V(T)$, the inclusion-minimal subtree of $T$ that contains $U$ is denoted by $\tree_T(U)$ or $\tree_T(u_1, u_2, \ldots, u_n)$. In other words, $\tree_T(U) = \bigcup_{u, v \in U} P_T(u, v)$.


\subsubsection*{Outline of paper.} In section \ref{sec:examples} we provide examples guaranteeing that the constants in Theorems~\ref{thm:seh} and~\ref{thm:ceh} cannot be increased (except for the class $\mathcal{T}_k$ whose construction will be given later). 
In section \ref{sec:lemmas} we prove two lemmas that will be useful in the proofs of both Theorems \ref{thm:seh} and \ref{thm:ceh}. The first one deals with ``generic'' subtree representations of chordal graphs and the other is a basic lemma on cographs. The proof of Theorem \ref{thm:seh} is given in section \ref{sec:seh}, and  section \ref{sec:ceh} contains the proof of  Theorem \ref{thm:ceh} as well as a probabilistic construction that shows that our bound for the class $\mathcal{T}_k$ is asymptotically tight. We wrap up in section \ref{sec:concl} with some final remarks and open problems.


\section{Optimality of constants in Theorem~\ref{thm:seh}~and~Theorem~\ref{thm:ceh}}\label{sec:examples}

\begin{example*}
The constants in Theorem \ref{thm:seh} can in general not be increased. Let us first consider the case of interval graphs. Let $G_1$ be the intersection graph of the intervals
\[I_1 = [0, 1], \; I_2 = [1, 2], \; I_3 = [2, 3], \; I_4 = [3, 4].\] Note that $G$ is a graph on four vertices and the largest bi-clique in $G$ or $\overline{G}$ has size two. We can make arbitrarily large examples by taking $k$ copies of each of the intervals. The resulting intersection graph has $4k$ vertices and the largest bi-clique in $G$ or $\overline{G}$ has size $2k$.
    
Next we give a construction for the case of cographs. Obviously a complete graph is a cograph, and so the graph $G = \overline{(K_k \cupdot K_k \cupdot K_k)} \cupdot K_k$ is a cograph on $4k$ vertices and it is easily checked that the largest bi-clique in $G$ or its complement $\overline{G}$ has size at most $2k$.

Finally, we give a construction for the case of chordal graphs by giving a subtree representation. Let $T = K_{1, 3}$ and $V(T) = \{v, u_1, u_2, u_3\}$ where $v$ is the vertex of degree 3.
Let $G$ be the intersection graph of the following nine subtrees of $T$:
\begin{align*}
T_1 = T_2 &= u_1,& T_3 = T_4 &= u_2,& T_5 = T_6 &= u_3,\\
T_7 &= P_T(u_1, u_2),& T_8 &= P_T(u_1, u_3),& T_9 &= P_T(u_2, u_3),
\end{align*}
where $P_T(u,v)$ denotes the unique path in $T$ connecting vertices $u$ and $v$. Thus $G$ is a graph on nine vertices and it is easily checked that the largest bi-clique contained in $G$ or $\overline{G}$ has size four. To obtain arbitrarily large examples simply take $k$ copies of each subtree. 

\end{example*}

\begin{example*}
The constants in Theorem \ref{thm:ceh} can in general not be increased. Here we give examples for the case of interval graphs and cographs. These are similar to the ones for the SEH-property. For the class $\mathcal{T}_k$ we will give an asymptotic matching bound, but the argument is a bit more involved and is given in section \ref{sec:ceh}. Note that this implies that the class of chordal graphs (which is the class $\mathcal{T}_\infty$) does not satisfy the CEH-property for any fixed constant $\varepsilon_c>0$. 

Here is a construction for interval graphs. 
Let $G$ be the intersection graph of the following intervals
\begin{align*}
I_1  &= [0, 1], & I_2 &= [2, 3],   & I_3 &= [4, 5],\\
J_1  &= [1, 2], & J_2 &= [3, 4], 
&J_3 &= [5, 6], 
\end{align*}
and consider the partition of the vertex set of $G$ containing the $I_m$ intervals in one part and the $J_m$ intervals in the other. Thus $G$ is a graph on six vertices and it is easily seen that the largest bi-clique in $G$ that respects this given partition has size two. As before we can make arbitrarily large example by taking $k$ copies of each of the intervals. 

Here is a construction for cographs.
Let $H$ denote the bipartite graph in the figure  below. \begin{figure}
    \centering
\begin{tikzpicture}
\begin{scope}
\coordinate (v1) at (0,0);
\coordinate (v2) at (1,0);
\coordinate (v3) at (2,0);
\coordinate (v4) at (3,0);
\coordinate (v5) at (0,1);
\coordinate (v6) at (1,1);
\coordinate (v7) at (2,1);
\coordinate (v8) at (3,1);
\fill (v1) circle (.3ex);
\fill (v2) circle (.3ex);
\fill (v3) circle (.3ex);
\fill (v4) circle (.3ex);
\fill (v5) circle (.3ex);
\fill (v6) circle (.3ex);
\fill (v7) circle (.3ex);
\fill (v8) circle (.3ex);
\draw (v1)--(v6);
\draw (v1)--(v7);
\draw (v1)--(v6);
\draw (v2)--(v5);
\draw (v2)--(v6);
\draw (v3)--(v5);
\draw (v3)--(v7);
\draw (v4)--(v8);
\node[below] at (v1) {\footnotesize $v_1$};
\node[below] at (v2) {\footnotesize $v_2$};
\node[below] at (v3) {\footnotesize $v_3$};
\node[below] at (v4) {\footnotesize $v_4$};
\node[above] at (v5) {\footnotesize $v_5$};
\node[above] at (v6) {\footnotesize $v_6$};
\node[above] at (v7) {\footnotesize $v_7$};
\node[above] at (v8) {\footnotesize $v_8$};
\end{scope}
\end{tikzpicture}
    \label{fig:my_label}
\end{figure}
Note that the largest bi-clique that respects the vertex partition of $H$ has size two.

Now consider the following cograph on vertices $\{v_1, \dots, v_8\}$. Let $G_1$ be the graph on $\{v_1, v_5\}$ without an edge, let 
$G_2$ be disjoint union of edges $v_2v_6$ and $v_3v_7$, and define $G_3 = \overline{G_1} \cupdot \overline{G_2}$. 
Finally, let $G_4$ be the edge $v_4 v_8$ and let $G_5 = \overline{G_3} \cupdot G_4$. It is easily seen that the induced bipartite subgraph between parts $\{v_1, \dots, v_4\}$ and $\{v_5, \dots, v_8\}$ is isomorphic to $H$.

For each $k \geq 1$, a general example $G$ on $8k$ vertices can be made by replacing each $v_i$ by \emph{any} cograph $H_i$ on $k$ vertices and partitioning $U_1 = \bigcup_{i = 1}^4 V(H_i)$ and $U_2 = \bigcup_{i = 5}^8 V(H_i)$. Note that $H_i$ can be any cograph because its edges disappear when we restrict $G$ to the edges between two parts $U_1$ and $U_2$.
\end{example*}


\section{Auxiliary results} \label{sec:lemmas}

We start this section with a simple lemma regarding subtree representations of chordal graphs. This will be useful later on in the proofs of Theorems~\ref{thm:seh} and~\ref{thm:ceh}. Recall that $\mathcal{T}_k$ is the class of chordal graphs with leafage at most $k$.

\begin{lemma} \label{lem:cubic}
    Any graph $G\in \mathcal{T}_k$ has a subtree representation as the intersection graph of a family $\{T_i\}$  of subtrees of an ambient tree $T$ where 
    \begin{enumerate}
        \item \label{lem:cubic-1} The ambient tree $T$ has $k$ leaves and maximum degree $3$.
        \item \label{lem:cubic-2} No two subtrees $T_i$ and $T_j$ share a common leaf.  
    \end{enumerate}
\end{lemma}

\begin{proof}
The fact that $G$ has a subtree representation in an ambient tree $T$ with at most $k$ leaves follows from the well-known result of Gavril \cite{Gavril1974} and the definition of $\ell(G)$. We fix such a subtree representation $\{T_i\}$ and show how to modify the ambient tree $T$ and the subtree representation without changing the intersection graph. 
    
We first show how to reduce the maximum degree of the ambient tree $T$. Fix a vertex $v \in V(T)$ of degree $d \geq 4$. Let $C_1, C_2, \ldots, C_d$ be the components of $T - v$,  and let  $u_i$ be the unique neighbor of $v$ in $C_i$ for each $i \in [d]$. Introduce new vertices $v_1$, $v_2$, $\dots$, $v_d$ connected by edges such that they form a path $P$.

We construct new tree $T'$ on the vertex set $(V(T) - \{v\}) \cup \{v_1, v_2, \ldots, v_d\}$. Edges in each $C_i$ remain the same in $T'$, each $u_i$ is connected to $v_i$ by an edge, and finally add the edges of $P$ to $T'$. In other words,
\[
E(T') = E(C_1) \cup E(C_2) \cup \cdots \cup E(C_d) \cup \{u_1 v_1, \ldots, u_d v_d\} \cup E(P).\] It is obvious that $T'$ is a tree.
Now we construct the new subtree family $\{T_i'\}$. If the original subtree $T_i$  contains the vertex $v$, then we set  $V(T_i')$ $= (V(T_i) - \{v\})$ $\cup$ $\{v_1$, $\ldots$, $v_d\}$. Otherwise we set 
$V(T_i') = V(T_i)$. Finally let $T_i'$ be the minimal subtree of $T'$ that contains the vertices $V(T_i')$. It is easily seen that $\{T_i\}$ and $\{T_i'\}$ have isomorphic intersection graphs. 

Note that the new vertices $v_1, \dots, v_d$ all have degree at most $3$ in $T'$, and that $T'$ has the same number of leaves as $T$. Therefore repeating the process until there are no more vertices of degree greater than 3 proves claim \eqref{lem:cubic-1}.

\smallskip

To prove claim \eqref{lem:cubic-2}, suppose
$v \in V(T)$ is a common leaf of $T_i$ and $T_j$. If $v$ is a leaf of $T$, then we first modify $T$ as follows; add a new vertex $u$ to $T$ and connect it to $v$. (We do not change any subtrees yet.) If $v$ is not a leaf of $T$, then $T$ remains unchanged.

Now let $u$ be a neighbor of $v$ which is not a vertex of $T_i$ (which must exist after possibly making the change above). Subdivide the edge $uv$ once; so the edge $uv$ is replaced by the path $uwv$ (in both $T$ and every subtree containing the original edge $uv$), then add the edge $vw$ to $T_i$.

In effect, this reduces the total number of common leaves between subtrees, while the intersection graph remains the same. We repeat the same procedure until no two subtrees share a common leaf.
\end{proof}

The next lemma concerning cographs will be needed for the proofs of both Theorems \ref{thm:seh} and \ref{thm:ceh}. Consider a vertex $v$ of a graph $G$ and subset $W\subset V(G)\setminus \{v\}$. We say that $W$ {\em conforms} to $v$ if either $W\subset N(v)$ or $W\cap N(v) = \emptyset$.

\begin{lemma}\label{lem:cograph}
Let $G$ be a cograph on the vertex set $V$. For any nonempty vertex set $U \subseteq V$, there exists a subset $W \subseteq V$ such that
\begin{enumerate}
    \item $\frac{1}{4}\size{U} \leq \size{U \cap W} \leq \max\{\frac{1}{2}\size{U}, 1\}.$
    \item For every $v \in V - W$, $W$ conforms to $v$.
\end{enumerate}
\end{lemma}

\begin{proof}
Let $\mathcal{P}$ denote the class of all cographs. Recall the inductive definition of $\mathcal{P}$:
\begin{enumerate}[(i)]
    \item The graph $K_1$ on one isolated vertex belongs to $\mathcal{P}$.
    \item If $G \in \mathcal{P}$, then its complement $\overline{G}$ also belongs to $\mathcal{P}$.
    \item If $G, H \in \mathcal{P}$, then their disjoint union $G \cupdot H$ also belongs to $\mathcal{P}$.
\end{enumerate}

For our given graph $G\in \mathcal{P}$ we inductively define sequences of cographs $\{G_i\}$ and $\{H_i\}$ as follows.
Start by setting $G_1 = G$. For $G_i$ with $i \geq 1$, either $G_i$ or $\overline{G_i}$ equals the disjoint union of two cographs $G_{i+1} \cupdot H_{i+1}$ by the inductive construction of $G_i$. Select $G_{i+1}$ so that $\size{U \cap V(G_{i+1})} \geq \size{U \cap V(H_{i+1})}$.
Since the order of $G_i$ is strictly decreasing in $i$, the sequence $\{G_i\}$ is finite and terminates when $G_m = K_1$ for some integer $m \geq 1$. This defines the two sequences $G_1, G_2, \ldots, G_m$ and $H_2, H_3, \ldots, H_m$.

\smallskip

Note that if $\size{U} = 1$, then we can prove the lemma simply by taking $W = V$. Thus we may assume $\size{U} \geq 2$.

\smallskip

If there exists an $i \geq 2$ such that $\frac{1}{4}\size{U} \leq \size{U \cap V(H_i)} \leq \frac{1}{2}\size{U}$, then pick the smallest such $i$ and set $W = V(H_i)$.

\smallskip

If no such $i$ exists, then we must have  $\size{U \cap V(H_i)}$ $< \frac{1}{4}\size{U}$ for every $2\leq i\leq m$. The case that $\size{U \cap V(H_i)}$ $> \frac{1}{2}\size{U}$ can not happen since this would contradict the disjointness of $U \cap V(G_i)$ and $U \cap V(H_i)$. 
Consider $\size{U \cap V(G_i)}$. We have $\size{U \cap V(G_1)} = \size{U}$ and $\size{U \cap V(G_m)} \leq 1$ on the other hand. Therefore, there is some $i$ such that $\size{U \cap V(G_i)} > \frac{1}{2}\size{U}$ but $\size{U \cap V(G_{i+1})} \leq \frac{1}{2}\size{U}$. Now we know that
\[\size{U \cap V(G_{i+1})} = \size{U \cap V(G_i)} - \size{U \cap V(H_{i+1})} > {\textstyle\frac{1}{2}n - \frac{1}{4}n = \frac{1}{4}n},
\]
and we set $W = V(G_{i+1})$.

\smallskip

It remains to show that for every $v \in V - W$, $W$ conforms to $v$. First consider the case $W = V(G_i)$ for some $i$. By the construction of the sequence $\{G_i\}$, we
have the series of inclusions $V(G_1) \supset V(G_2) \supset \cdots \supset V(G_{i-1}) \supset V(G_{i})$. Let $j$ be the largest index such that $v \in V(G_j)$. Obviously $1 \leq j < i$, and since $V(G_j) = V(G_{j+1}) \cupdot V(H_{j+1})$, the choice of $j$ implies that $v \in V(H_{j+1})$. As $G_{j+1} \cupdot H_{j+1}$ equals $G_j$ or $\overline{G_j}$, 
$V(G_{j+1})$ conforms to any vertex in $V(H_{j+1})$.
This shows the desired property since $W = V(G_i) \subseteq V(G_{j+1})$.

The remaining case $W = V(H_i)$ for some $i$ can be proved in a similar way, since every $v \in V - W$ is either in $V(G_i)$ or in $V(H_j)$ for some $1 \leq j < i$.
\end{proof}


\section{The strong Erd{\H o}s--Hajnal property} \label{sec:seh}

In this section we prove Theorem \ref{thm:seh}. The cases of interval graphs and cographs are simple and will be treated first. The case of chordal graphs is more involved and takes up the majority of the proof.

\begin{proof}[Proof of Theorem \ref{thm:seh} for interval graphs]
Let $G$ be an interval graph on $n$ vertices. 
By Lemma \ref{lem:cubic} we may assume $G$ has a representation as the intersection graph of a family $\{I_j\}_{j=1}^n$ of compact intervals on the real line such that no two intervals share a common endpoint. 
For a point $x\in \mathbb{R}$, let $L(x)$ denote the number of intervals $I_j$ whose rightmost endpoint is strictly less than $x$, and let $R(x)$ denote the number of intervals whose leftmost endpoint is strictly greater than $x$. 
Observe that for all sufficiently small $x$ we have $L(x) = 0$ and $R(x) = n$, and for all sufficiently large $x$ we have $L(x) = n$ and $R(x) = 0$. 
Since the intervals all have distinct endpoints it follows that $L(x)$ is weakly increasing and changes in increments of 1, 
while $R(x)$ is weakly decreasing and changes in increments of -1. Moreover these changes happen at distinct $x$-values.
Therefore there exists a point $x_0$ such that $M = L(x_0) = R(x_0)$. If $M\geq n/4$ then there is a bi-clique in $\overline{G}$ of size $2\lfloor n/4 \rfloor$. If $M < n/4$ then $G$ contains a clique of size $n/2$ and consequently $G$ contains a bi-clique of size $2\lfloor n/4 \rfloor$.
\end{proof}

\begin{proof}[Proof of Theorem \ref{thm:seh} for cographs]
We apply Lemma~\ref{lem:cograph} to $U = V(G)$. This gives us a subset $W \subseteq V$ with $\frac{1}{4}n \leq \size{W} \leq \frac{1}{2}n$ such that $W$ conforms to every $v \in V - W$. Note that $\size{V - W} \geq \frac{1}{2}n$, and so therefore there exists a subset $X\subset V-W$ with $|X|\geq n/4$ such that either $W\subset N(x)$ for every $x\in X$, or such that $W\cap N(x) = \emptyset$ for every $x\in X$.  This implies $G$ or $\overline{G}$ contains a bi-clique of desired size.
\end{proof}

\begin{proof}[Proof of Theorem \ref{thm:seh} for chordal graphs] Let $G$ be a chordal graph on $n$ vertices, and for contradiction, we assume that neither $G$ nor $\overline{G}$ contains a bi-clique of size $2\left\lfloor\frac{2}{9}n\right\rfloor$.
By Lemma \ref{lem:cubic} we may assume that $G$ has a subtree representation as an intersection graph of a family of subtrees $F = \{T_i\}_{i=1}^n$ of an ambient tree $T$ where $\Delta(T) \leq 3$ and no two subtrees $T_i$ and $T_j$ share a common leaf.
We may assume the maximum degree $\Delta(T) =3$, otherwise $T$ is a path (or possibly a single vertex) and therefore $G$ is an interval graph which was treated above.

For each $v \in V(T)$, let $F_v\subset F$ be the collection of subtrees that contain the vertex $v$. If $\size{F_v} \geq \frac{4}{9}n \geq 2\left\lfloor\frac{2}{9}n\right\rfloor$ for some $v$, then the members of $F_v$ form a clique in $G$ and we are done. Therefore assume $\size{F_v} < \frac{4}{9}n$ for every $v$.

For any vertex $v\in V(T)$,  let $C_i(v)$ be the components of $T - v$, and let $F_i(v)\subset F$ be the family of subtrees contained in $C_i(v)$. 
(Note that we allow for the possibility that some of the  $C_i(v)$ and/or $F_i(v)$ are empty.)
Clearly, for every vertex $v\in V(T)$ we have $F = F_v \cupdot F_1(v) \cupdot F_2(v) \cupdot F_3(v)$. 

\begin{claim}\label{c:subtreespt}
There exists a degree 3 vertex $v \in V(T)$ such that $\frac{1}{9}n \leq \size{F_i(v)} \leq \frac{2}{9}n$ for every $i = 1,2,3$.
\end{claim}

\begin{proof}[Proof of Claim~\ref{c:subtreespt}]
For every vertex $v$ let us label the components $C_i(v)$ such that $|F_1(v)|\geq |F_2(v)|\geq |F_3(v)|$. 
We first show if there is no vertex that satisfies the claim, then we have   $\size{F_1(v)} \geq \frac{2}{9}n > |F_2(v)| \geq |F_3(v)|$ for every vertex $v\in V(T)$. 
To see why, assume there is a vertex $v$ such that $\size{F_1(v)} < \frac{2}{9}n$. If $\deg(v) < 3$, then 
\[{\textstyle{\size{F} = \size{F_v} + \size{F_1(v)} + \size{F_2(v)} + \size{F_3(v)} \leq \frac{4}{9}n + \frac{2}{9}n + \frac{2}{9}n + 0 < n = \size{F}}},\]
which is a contradiction. Therefore we have $\deg(v) = 3$, and we get 
\[
\size{F_i(v)} \geq |F_3(v)| = \size{F} - \size{F_v} - |F_1(v)| - |F_2(v)|
> { \textstyle n - \frac{4}{9}n - \frac{2}{9}n - \frac{2}{9}n = \frac{1}{9}n},\]
but then $v$ is a vertex satisfying the claim. Consequently we must have $|F_1(v)|\geq \frac{2}{9}n$. If also $\size{F_2(v)} \geq \frac{2}{9}n$, then $F_1(v)$ and $F_2(v)$ correspond to a bi-clique in $\overline{G}$ of the desired size, and therefore  
$|F_1(v)|\geq \frac{2}{9}n > |F_2(v)| \geq |F_3(v)|$.

\smallskip

Now consider the following orientation of the edges of $T$. For any given $v\in V(T)$, let $u$ be the (unique) neighbor of $v$ contained in $C_1(v)$ that is adjacent $v$, and assign the orientation $\overrightarrow{vu}$. 
By the observations in the previous paragraph, every vertex has a {\em unique} outgoing edge. Furthermore we claim that every edge will be assigned a {\em unique} orientation. This is because if an edge $uv \in E(T)$ is assigned either no orientation or both orientations, then $C_1(v)$ and $C_1(u)$ are disjoint, which implies that $F_1(v)$ and $F_1(u)$ correspond to a bi-clique in $\overline{G}$ of the desired size. 

Thus, if there is no vertex satisfying the claim, then we obtain an orientation of the edges of $T$ in which each vertex has a unique outgoing edge, which is impossible. 
%
%
%
\renewcommand{\qedsymbol}{$\blacksquare$}
\end{proof}

Now fix a vertex $v \in V(T)$ satisfying the condition in Claim~\ref{c:subtreespt}, and let $u_i$ denote the unique neighbor of $v$ in the component $C_i(v)$. For every $w \in V(T) - \{v\}$, we label the components $C_i(w)$ of $T - w$ (some of which may be empty) such that $C_1(w)$ contains the vertex $v$, and  define $C_{23}(w)$ as the induced subgraph of $T$ on $V(T) - V(C_1(w))$.
Define $\Gamma(w)$ to be the collection of subtrees in $F$ which are contained in either $C_2(w)$ or $C_3(w)$, that is, $\Gamma(w) = F_2(w) \cup F_3(w)$, and let $\Gamma^+(w)$ be the collection of subtrees in $F$ contained in $C_{23}(w)$. Note that according to this new notation, we have $\Gamma^+(u_i) = F_i(v)$ for every $i = 1,2,3$.


\begin{claim}\label{c:subtree1/9}
For every $i = 1,2,3$, there is a vertex $w_i \in V(C_i(v))$ that satisfies the following:
\begin{enumerate}[(i)]
    \item $\size{\Gamma^+ (w_i)} \geq \frac{1}{9}n$
    \item $\size{F_2(w_i)}, \size{F_3(w_i)} < \frac{1}{9}n$.
\end{enumerate}
\end{claim}

\begin{proof}[Proof of Claim~\ref{c:subtree1/9}]
We argue for the case $i=1$. (The other cases follow by the same argument.)  By the choice of $v$ we have $\frac{1}{9}n \leq |\Gamma^+(u_1)| \leq \frac{2}{9}n$,  and we are done if both $\size{F_2(u_1)}$ and $\size{F_3(u_1)}$ are strictly less than $\frac{1}{9}n$. Otherwise, we may assume that $|F_2(u_1)| > \frac{1}{9}n > |F_3(u_1)|$ and we repeat the argument for the unique neighbor of $u_1$ in the component $C_2(u_1)$. We can repeat this process until it eventually terminates at the desired vertex $w_1$. 
\renewcommand{\qedsymbol}{$\blacksquare$}
\end{proof}

For each $i = 1,2,3$, let $G_i \subset F_i(v)$ be the collection of subtrees that intersect the path $P_T(v, w_i)$. Let $H_i\subset F_i(v)$ be the collection of subtrees that are disjoint from the subgraph $C_{23}(w_i)$ and from the path $P_T(v, w_i)$. Equivalently, $H_i = (F_i(v) - G_i) - \Gamma^+(w_i)$. 
Note that $\size{H_i} \leq \frac{1}{9}n$ since $H_i$ and $\Gamma^+(w_i)$ are disjoint subfamilies of $F_i(v)$ and  $\frac{1}{9}n \leq \size{\Gamma^+(w_i)} \leq  \size{F_i(v)} \leq \frac{2}{9}n$.

\smallskip


Finally, we define some additional subfamilies of $F_v$. Let $X_\emptyset \subset F_v$ be the collection of subtrees that contains neither of $w_1, w_2, w_3$. Let $X_i \subset F_v$ be the collection of subtrees containing only $w_i$ but {\em not} the other $w_j$'s. For every $1 \leq i < j \leq 3$, let $Y_{ij} \subset F_v$ be the collection of subtrees containing both $w_i$ and $w_j$. Note that a member of $F$ that contains $w_1, w_2$ and $w_3$, belongs to $Y_{12}, Y_{13}$, and $Y_{23}$.

\smallskip

The following observations identify certain bi-cliques in $G$ or $\overline{G}$ which allow us to bound the sizes of the various subfamilies we have defined. This will eventually lead us to the existence of a bi-clique of size $2\left\lfloor\frac{2}{9}n\right\rfloor$ in $G$ or $\overline{G}$.

\begin{obs} 
Every member of $X_\emptyset \cup X_1 \cup F_1(v)$ is disjoint from every member of $\Gamma^+(w_2) \cup \Gamma^+(w_3)$, and so the two subfamilies correspond to a bi-clique in $\overline{G}$.
\end{obs}
\noindent The same obviously holds for the symmetric cases as well. 
So by the assumption that $G$ or $\overline{G}$ contains no bi-clique of size $2\left\lfloor\frac{2}{9}n\right\rfloor$, and since 
$\Gamma^+(w_i) \geq \frac{1}{9}n$,  we must have
\[{\textstyle{
\size{X_\emptyset} + \size{X_i} + \size{F_i(v)} < \frac{2}{9}n, }}
\] for every $i = 1,2,3$. We now have the following inequality.
\begin{align*} 
    \size{X_\emptyset} + \size{X_1} + \size{X_2} + \size{X_3} & \leq 3\size{X_\emptyset} + \size{X_1} + \size{X_2} + \size{X_3} \\
    & < {\textstyle {\left(\frac{2}{9}n - \size{F_1(v)} \right) + \left(\frac{2}{9}n - \size{F_2(v)} \right) + \left(\frac{2}{9}n - \size{F_3(v)} \right)}}\\
    & = {\textstyle{\frac{2}{3}n - (\size{F_1(v)} + \size{F_2(v)} + \size{F_3(v)}).}}
\end{align*}

\noindent Note that $F_v = X_\emptyset \cupdot X_1 \cupdot X_2 \cupdot X_3 \cupdot (Y_{12} \cup Y_{13} \cup Y_{23})$, and so we have
\begin{align*}
    \size{Y_{12} \cup Y_{13} \cup Y_{23}} &= \size{F_v} - (\size{X_\emptyset} + \size{X_1} + \size{X_2} + \size{X_3})\\
    &= n - (\size{F_1(v)} + \size{F_2(v)} + \size{F_3(v)}) - (\size{X_\emptyset} + \size{X_1} + \size{X_2} + \size{X_3})\\
    &> {\textstyle{\frac{1}{3}n.}}
\end{align*}

\noindent By double-counting, one of  $\size{Y_{12} \cup Y_{13}}$, $\size{Y_{12} \cup Y_{23}}$, or $\size{Y_{13} \cup Y_{23}}$ is strictly greater than $\frac{2}{9}n$, and without loss generality we may assume that $\size{Y_{12} \cup Y_{13}} > \frac{2}{9}n$. Choose a subset $Y \subseteq Y_{12} \cup Y_{13}$ of size $\lfloor\frac{2}{9}n\rfloor$.

\begin{obs}
Every member of $Y$ intersects every member of $(F_v - Y) \cup G_1$, and so the two subfamilies correspond to a bi-clique in $G$.
\end{obs}

\noindent We may therefore assume that $\size{F_v - Y} + \size{G_1} < \frac{2}{9}n$, which gives us
\begin{align*}
    {\textstyle\frac{2}{9}n} & > \size{F_v - Y} + \size{G_1} = \size{F_v} - {\textstyle\lfloor\frac{2}{9}n\rfloor} + \size{G_1}\\
    &\geq n - (\size{F_1(v)} + \size{F_2(v)} + \size{F_3(v)}) - {\textstyle\frac{2}{9}n} + \size{G_1}\\
    &= {\textstyle\frac{7}{9}n} - (\size{F_2(v)} + \size{F_3(v)}) - \size{F_1(v)} + \size{G_1}.
\end{align*}

\noindent From this we can conclude that 
\begin{align*}
    \size{H_1} + \size{F_2(w_1)} + \size{F_3(w_1)} &= \size{F_1(v)} - \size{G_1}\\
    & > {\textstyle\frac{5}{9}n} - (\size{F_2(v)} + \size{F_3(v)}),
\end{align*}

\noindent and therefore
\[
\size{H_1} + \size{F_2(w_1)} + \size{F_3(w_1)} + \size{F_2(v)} + \size{F_3(v)} > {\textstyle\frac{5}{9}n.
}\]

\begin{obs}
Any two members taken from distinct families among $H_1$, $F_2(w_1)$, $F_3(w_1)$, $F_2(v)$, $F_3(v)$ are pairwise disjoint. 
\end{obs}

\noindent In particular, if we partition $S = \{H_1, F_2(w_1), F_3(w_1), F_2(v), F_3(v)\}$ into two parts $S = S_1 \cupdot S_2$, then this corresponds to a bi-clique in $\overline{G}$. 
Our final goal is to divide $S$ evenly so that both $\bigcup S_1 := \bigcup_{G\in S_1}G$ and $\bigcup S_2:= \bigcup_{G\in S_2}G$ each contain in total at least $\frac{2}{9}n$ subtrees.

Recall $\frac{1}{9}n \leq \size{F_2(v)}, \size{F_3(v)} \leq \frac{2}{9}n$ and all three subfamilies $H_1$, $F_2(w_1)$, and $F_3(w_1)$ have size at most $\frac{1}{9}n$. Now we describe how to split $S$ evenly. Start with $S_1 = \{F_2(v)\}$ and $S_2 = \{F_3(v)\}$. Next, take one of the remaining subfamilies in $S - (S_1 \cup S_2)$ and add it to the part $S_i$ which contains the fewest total number of subtrees. Repeat this until the three subfamilies $H_1$, $F_2(w_1)$, $F_3(w_1)$ have been distributed. Then the resulting $S_1$ and $S_2$ satisfy $\size{\size{\bigcup S_1} - \size{\bigcup S_2}} \leq \frac{1}{9}n$, since $\size{\size{F_2(v)} - \size{F_3(v)}} \leq \frac{1}{9}n$ and as we distribute the remaining subfamilies, the difference $\size{\bigcup S_1} - \size{\bigcup S_2}$ changes by at most $\frac{1}{9}n$ in each step. Because $\size{\bigcup S_1} + \size{\bigcup S_2} > \frac{5}{9}n$, we have that $\bigcup S_1$ and $\bigcup S_2$ both contain at least $\frac{2}{9}n$ subtrees, which completes the proof. \end{proof}


\section{The colorful Erd{\H o}s--Hajnal property} \label{sec:ceh}

In this section we prove Theorem \ref{thm:ceh}. As in the previous section, the cases of interval graphs and cographs are simple and will be treated first. Finally we deal with the case of the graph class $\mathcal{T}_k$, where we give asymptotically matching upper and lower bounds. 

\begin{proof}[Proof of Theorem \ref{thm:ceh} for interval graphs]
Let $G$ be an interval graph on $n$ vertices. By Lemma~\ref{lem:cubic} $G$ has a representation as the intersection graph of a family $F$ of $n$ compact intervals on the real line such that no two intervals share a common endpoint. 

Our goal is to show that for any partition $F = F_1 \cup F_2$ such that $|F_1| = \left\lceil \frac{n}{2} \right\rceil$ and $|F_2| = \left\lfloor \frac{n}{2} \right\rfloor$ there are subfamilies of $H_1 \subset F_1$ and $H_2\subset F_2$, each of size at least $\left\lfloor\frac{n}{6}\right\rfloor$, such that either every member of $H_1$ intersects every member of $H_2$, or every member of $H_1$ is disjoint from every member of $H_2$.

For each $i = 1, 2$, let $a_i$ be the smallest real number such that at least one third of the members of $F_i$ are  contained in the half-line $(-\infty, a_i]$:
\[
a_i := \min \left\{a \in \R: \size{\{I \in F_i : I \subseteq (-\infty, a] \}} \geq {\textstyle\frac{1}{3}} |F_i| \right\} .\]
Similarly, define $b_i$ as the largest real number such that at least one third of elements of $F_i$ are contained in the half-line $[b_i, \infty)$:
\[b_i := \max \left\{b \in \R: \size{\{I \in F_i : I \subseteq [b, \infty) \}} \geq \textstyle{\frac{1}{3}} |F_i| \right\}.\]
For an interval $J \subset \mathbb{R}$, let
$\left. F_i \right|_{J} \subset F_i$  denote the collection of intervals of $F_i$ that are contained in $J$. Note that both
$\left.F_i \right|_{(-\infty, a_i]}$ and $\left.F_i\right|_{[b_i, \infty)}$ have size exactly $\left\lceil\frac{1}{3}|F_i|\right\rceil$, and both $\left.F_i \right|_{(-\infty, a_i)}$ and $\left.F_i\right|_{(b_i, \infty)}$ have size exactly $\left\lceil\frac{1}{3}|F_i|\right\rceil-1$.

\medskip

\noindent We divide cases according to the relative order between $a_1, a_2, b_1$ and $b_2$.

\medskip

\noindent\emph{Case 1.} $a_1 < b_2$ or $a_2 < b_1$ : 
If $a_1 < b_2$, then we set $H_1 = \left.F_1 \right|_{(-\infty, a_1]}$ and $H_2 = \left.F_2\right|_{[b_2, \infty)}$ to obtain the desired subfamilies corresponding to a bi-clique in $\overline{G}$. 
The case $a_2 < b_1$ is symmetric handled in the same way.

\medskip

\noindent For the rest of proof, assume $b_1 \leq a_2$ and $b_2 \leq a_1$. Note that these conditions  imply that we must have $b_1 \leq a_1$ or $b_2 \leq a_2$, and by symmetry we may assume that $b_2 \leq a_2$ holds.

\medskip

\noindent\emph{Case 2.} $a_1 < b_1$ and $b_2 \leq a_2$ :
We have $b_2 \leq a_1 < b_1 \leq a_2$.
In this case we set $H_1 =F_1 - \left( \left. F_1 \right|_{(-\infty, a_1)} \cup \left. F_1 \right|_{(b_1, \infty)} \right)$
and $H_2 =F_2 - \left( \left. F_2 \right|_{(-\infty, a_2)} \cup \left. F_2 \right|_{(b_2, \infty)} \right)$. 
Then $|H_i| \geq |F_i|-2\left\lceil\frac{1}{3}|F_i|\right\rceil+2 \geq \left\lceil\frac{1}{3}|F_i|\right\rceil \ge \left\lfloor\frac{n}{6}\right\rfloor$.
Observe that every interval $I\in H_1$ must {\em intersect} the interval $[a_1,b_1]$, and that every interval in $H_2$ must {\em contain} the interval $[b_2,a_2]$. Since $[a_1,b_1]\subset [b_2,a_2]$ it follows that the families $H_1$ and $H_2$ are the desired subfamilies corresponding to a bi-clique in $G$. 

\medskip

\noindent\emph{Case 3.} $b_1 \leq a_1$ and $b_2 \leq a_2$ : Note that in this case, the intervals $[b_1, a_1]$ and $[b_2, a_2]$ must intersect or else we are in the situation of Case 1. 
Here we set $H_i =F_i - \left( \left. F_i \right|_{(-\infty, a_i)} \cup \left. F_i \right|_{(b_i, \infty)} \right)$ for both $i = 1, 2$. As in Case 2, we have $\size{H_i} \geq \left\lceil\frac{1}{3}|F_i|\right\rceil \ge \left\lfloor\frac{n}{6}\right\rfloor$ and every member in $H_i$ must contain the interval $[b_i,a_i]$, and consequently the families $H_1$ and $H_2$ correspond to a bi-clique in $G$.
\end{proof}

\begin{proof}[Proof of Theorem \ref{thm:ceh} for cographs]
Let $G$ be a cograph on the vertex set $V$ with a partition $V = V_1 \cup V_2$ where $|V_1| = \lceil |V|/2 \rceil$ and $|V_2| = \lfloor |V|/2 \rfloor$. Our goal is to find subsets $U_i\subset V_i$ of size at least $|V_i|/4$ for each $i=1,2$ such that either every vertex in $U_1$ is adjacent to all vertices of $U_2$ or there is no adjacent pair of $u_1 \in U_1, u_2 \in U_2$.
This implies the existence of a bi-clique of desired size in $G$ or $\overline{G}$.

Applying Lemma~\ref{lem:cograph} with $U = V_1$, we get a subset $W \subseteq V$ such that $\frac{1}{4}\size{V_1} \leq \size{V_1 \cap W} \leq \frac{1}{2}\size{V_1}$ and $W$ conforms to every $v \in V - W$.
Now define subsets $X(W), Y(W) \subseteq V - W$ by setting
\[
\renewcommand*{\arraystretch}{1.4}
\begin{array}{rcl}
X(W) & := & \{v \in V - W : W \subset N(v)\},   \\
Y(W) & := & \{v \in V - W : W\cap N(v) = \emptyset\}.
\end{array}\]
By the choice of $W$ using Lemma~\ref{lem:cograph}, we have $X(W) \cupdot Y(W) = V - W$. We now distinguish two cases.

\medskip

\noindent \emph{Case 1.} $\size{V_2 \cap W} < \frac{1}{2}\size{V_2}$:
Note that $\size{V_2 \cap (V - W)} \geq \frac{1}{2}\size{V_2}$. Then one of the sets $V_2 \cap X(W)$ or $V_2 \cap Y(W)$ has cardinality at least $\frac{1}{4}\size{V_2}$. We define $U_2$ to be the set of larger cardinality, and define $U_1 = V_1 \cap W$.

\medskip

\noindent \emph{Case 2.} $\size{V_2 \cap W} \geq \frac{1}{2}\size{V_2}$:
Note that $\size{V_1 \cap (V - W)} \geq \frac{1}{2}\size{V_1}$. Then one of the sets $V_1 \cap X(W)$ or $V_1 \cap Y(W)$ has cardinality at least $\frac{1}{4}\size{V_1}$. We define $U_1$ to be the set of larger cardinality, and defined by $U_2 = V_2 \cap W$.

\medskip

In both cases, we have
$\size{U_i} \geq \size{V_i}/4$ for $i = 1, 2$, and this completes the proof.
\end{proof}

\begin{proof}[Proof of Theorem \ref{thm:ceh} for chordal graphs]
The case $k = 2$ is covered by results on intervals. We therefore assume that $G$ is a chordal graph on $n$ vertices, with leafage $\ell(G) = k\geq 3$, and suppose we are given a partition of the vertices into subsets $V_1$ and $V_2$ whose sizes differ by at most one.

By Lemma \ref{lem:cubic} we may fix a subtree representation of $G$ as an intersection graph of a family $F$ of $n$ subtrees of an ambient tree $T$ where $T$ has $k$ leaves, $\Delta(T) = 3$,  and where no two subtrees share a common leaf. The vertex partition $V_1\cup V_2$ corresponds to a partition $F = F_1 \cup F_2$.

Let $L(T) = \{v_1, v_2, \ldots, v_k\}$ be the $k$ leaves of $T$. For each $i \in [k]$, define $L_i = L(T) - \{v_i\}$ and let $T_i = \tree_T(L_i)$ which is $T$ with one leaf removed. Let $r_i$ be the closest degree 3 vertex to $v_i$ so that $E(T) - E(T_i) = E(P_T(v_i, r_i))$. Let $s_i$ be the unique neighbor of $r_i$ in $P_T(v_i, r_i)$.

First we show the CEH-property of $\mathcal{T}_k$ with a smaller constant:

\begin{claim}\label{claim:weakerceh}
$\mathcal{T}_k$ satisfies the CEH-property with constant $\varepsilon_c = \frac{1}{3(k-1)}$.
\end{claim}

\begin{proof}[Proof of Claim~\ref{claim:weakerceh}]
Again, the case $k = 2$ is covered by results on intervals. We proceed to induction on $k$. Assume that $k \geq 3$ and let $T, F_1, F_2$ and $T_i$'s be as above.

\smallskip

If there is some $i$ such that $T_i$ intersects both at least $\ceil{\frac{k - 2}{k - 1}\size{F_1}}$ members of $F_1$ and $\ceil{\frac{k - 2}{k - 1}\size{F_2}}$ members of $F_2$, then the result follows by applying the inductive assumption on families
$$
F_j(T_i) := \{X \cap T_i : X \in F_j\}
$$
with $j = 1, 2$.

\smallskip

Therefore, for each $i \in [k]$ the path $P_T(v_i, s_i)$ contains either at least $\ceil{\frac{1}{k - 1}\size{F_1}}$ members of $F_1$ or $\ceil{\frac{1}{k - 1}\size{F_2}}$ members of $F_2$. Assume that $P_T(v_1, s_1)$ contains at least $\ceil{\frac{1}{k - 1}\size{F_1}}$ members of $F_1$. If $T_1$ contains $\ceil{\frac{1}{k - 1}\size{F_2}}$ members of $F_2$, then the two subfamilies correspond a bi-clique of desired size in $G$.

Otherwise, at least $\ceil{\frac{k - 2}{k - 1}\size{F_2}}$ members of $F_2$ make nonempty intersections with $P_T(v_1, s_1)$. Consider the intersections as subpaths of $P_T(v_1, s_1)$. Together with $\ceil{\frac{1}{k - 1}\size{F_1}}$ members of $F_1$ contained in the same path, by Theorem~\ref{thm:ceh} (1), there exists a bi-clique of desired size in $G$ or $\overline{G}$.
\renewcommand{\qedsymbol}{$\blacksquare$}
\end{proof}

Now we prove the CEH-property for $\mathcal{T}_k$ with the promised constant $\varepsilon_c = \frac{\ln k}{20k}$. Let $T$ and $r_1, r_2, \ldots, r_k$ as above. We define $\tree_T(r_1, r_2, \ldots, r_k)$ as the \emph{trunk} of $T$, denoted by $\trunk(T)$.

A key observation is that the trunk of a tree is again a tree with fewer leaves:

\begin{obs}\label{o:LeavesofTrunk}
$\trunk(T)$ is a tree with at most $\floor{\frac{k}{2}}$ leaves.
\end{obs}
\begin{proof}[Proof of Observation~\ref{o:LeavesofTrunk}]
Let $\ell$ be a leaf of $\trunk(T)$. We claim that $\ell = r_i$ for at least two indices $i \in [k]$ which is sufficient to prove the observation.

First we show that $\ell = r_i$ for some $i$. Note that $\trunk(T) = \tree_T(r_1, r_2, \ldots, r_k) = \bigcup_{i, j \in [k]} P_T(r_i, r_j)$. Hence $\ell \in P_T(r_i, r_j)$ for some $i, j \in [k]$. However the degree of $\deg(\ell) = 1$ so it cannot be an interior vertex of the path. Thus $\ell$ is either $r_i$ or $r_j$.

Next, we show that $\ell = r_i$ for at least two values of $i$. If $\trunk(T)$ has no edge then $T$ is a subdivision of star $K_{1, k}$ and we are done. Now consider the case where $\trunk(T)$ has an edge. Let $w$ be the unique neighbor of $\ell$ in $\trunk(T)$. Since $\ell = r_i$, it has another neighbor $s_i$ in $T$, which is the unique neighbor of $\ell$ in $P_T(v_i, \ell) = P_T(v_i, r_i)$.
Since $\deg(r_i) = 3$ in $T$, $\ell = r_i$ has the third neighbor $u$ in $T$ other than $w$ and $s_i$.

Consider a leaf $v_j$ of $T$ such that $P_T(v_j, \ell)$ contains $u$. Note that it must be $j \neq i$, and we finish the proof by showing $\ell = r_j$. Assume not. Then the path $P_{\trunk(T)}(\ell, r_j)$ contains $w$ since $w$ is a unique neighbor of $\ell$ in $\trunk(T)$. However it implies $w \in P_T(\ell, r_j)$. On the other hand, $P_T(v_j, \ell)$ and $P_T(\ell, r_j)$ are edge disjoint hence $P_T(v_j, \ell) \cup P_T(\ell, r_j)$ is a path from $v_j$ to $r_j$ in $T$. Since $T$ is tree, it is also the unique path between $v_j$ and $r_j$. But this contradicts to that $r_j$ is the closest vertex of degree 3 to $v_j$.
\renewcommand{\qedsymbol}{$\blacksquare$}
\end{proof}

Let $R = \trunk(T) \subset T$. For each $i = 1, 2$, define $\left( F_i \right)_R$ to be the collection of subtrees in $F_i$ that intersect $R$:
\[
\left( F_i \right)_R := \{ X \in F_i : X \cap R \neq \emptyset \}.
\]

The complement of $\left( F_i \right)_R$, which is the collection of subtrees in $F_i$ that are disjoint from $R$ is denoted by $\overline{\left( F_i \right)_R}$:
\[
\overline{\left( F_i \right)_R} := \{ X \in F_i : X \cap R = \emptyset \}.
\]

Thus, each member of $\overline{\left( F_i \right)_R}$ can be viewed as a subpath of $P_T(v_j, s_j)$ for some $j \in [k]$, where $s_j$ is the unique neighbor of $r_j$ in $P_T(v_j, r_j)$.

We will prove the theorem using induction on $k$. We divide into cases according to the size of $\left( F_1 \right)_R$ and $\left( F_2 \right)_R$.

First, assume that both $\left(F_1 \right)_R$ and $\left( F_2 \right)_R$ are big, say $\size{\left( F_i \right)_R} \geq \frac{2}{3}|F_i|$ for each $i = 1, 2$. Define subtree families $F_i(R) := \{X \cap R : X \in \left( F_i \right)_R \}$ of $R$ as a multiset. By the induction hypothesis there exists subfamilies $\left( F_1 \right)_R' \subseteq F_1(R)$ and $\left( F_2 \right)_R' \subseteq F_2(R)$ corresponding to a bi-clique in $G$ or $\overline{G}$, with size:
\[
\size{\left( F_i \right)_R'} \geq \textstyle \frac{1}{20} \frac{\ln (k/2)}{k/2} \size{F_i(R)} \geq \frac{1}{20} \frac{\ln (k/2)}{k/2} \frac{2}{3} |F_i| \geq \frac{1}{20} \frac{\ln k}{k} |F_i|,
\]
where the first inequality comes once we think of $R$ as a tree of at most $\floor{k/2}$ leaves.

\smallskip

Next, consider the case where both $\left( F_1 \right)_R$ and $\left( F_2 \right)_R$ are small, meaning $\size{\left( F_i \right)_R} < \frac{2}{3}|F_i|$ for each $i$. Then we have $| \overline{\left( F_i \right)_R} | \geq \frac{1}{3}|F_i|$. Recall that each element of $\overline{\left( F_i \right)_R}$ is a subpath of some path $P_T(v_j, s_j)$. Therefore we may view each family $\overline{(F_i)_R}$ as a family of intervals contained in the open interval $(i-1, i) \subseteq \R$, and Theorem~\ref{thm:ceh} for interval graphs guarantees the existence of subfamilies $F'_i \subseteq \overline{\left( F_i \right)_R}$ of size $| F'_i | \geq \frac{1}{3} |\overline{\left( F_i \right)_R}| \geq \frac{1}{9} |F_i|$.

\smallskip

Finally, consider the last case where only one of $\left( F_1 \right)_R$ or $\left( F_2 \right)_R$ is big. Without loss of generality assume $\left( F_1 \right)_R$ is big so that $\size{\left( F_1 \right)_R} \geq \frac{2}{3}|F_1|$ and $| \overline{\left( F_2 \right)_R}| \geq \frac{1}{3}|F_2|$. 
For each $j \in [k]$, define the family $H_j \subset \overline{\left( F_2 \right)_R}$ of subtrees  contained in $P_T(v_j, s_j)$. Note that $\{H_1, H_2, \ldots, H_k\}$ form a partition of $\overline{\left( F_2 \right)_R}$. Assume the size of parts are in decreasing order so that $\sum_{i \leq m} \size{H_i} \geq \frac{m}{k} | \overline{\left( F_2 \right)_R} |$.

We will choose two sequences of subfamilies 
$\left( F_1 \right)_R = F^{(0)} \supseteq F^{(1)} \supseteq \ldots \supseteq F^{(k)}$ and $H'_1 \subseteq H_1, H'_2 \subseteq H_2, \ldots, H'_k \subseteq H_k$ satisfying the following three conditions for every $j \in [k]$:

\begin{enumerate}[(i)]
    \item $\size{F^{(j)}} \geq \frac{1}{2} \size{F^{(j-1)}} \geq \frac{1}{2^j} \size{\left( F_1 \right)_R}$.
    \item $\size{H'_j} \geq \frac{1}{2}\size{H_j}$.
    \item $F^{(j)}$ and $H'_j$ correspond to a bi-clique in $G$ or $\overline{G}$.
\end{enumerate}

Let us first show how the existence of such subfamilies yields the conclusion of theorem. Note that for every $j \leq m \leq k$, $F^{(m)} \subseteq F^{(j)}$ and $H'_j$ also correspond to bi-cliques. For a fixed $m \in [k]$ we produce a partition of $[m]$ as follows: Let 
$I\subset [m]$ denote the set of  indices $i\in [m]$ such that every member of $F^{(m)}$ intersects every member of $H'_i$. Similarly, let $J\subset [m]$ denote the set of indices $j\in [m]$ such that every member of $F^{(m)}$ is disjoint from every member of $H'_j$.
Note that $H'_I = \bigcup_{i \in I} H'_i$ and $H'_J = \bigcup_{j \in J} H'_j$ are disjoint, hence one of them has size at least one half of $|{\bigcup_{i \in [m]} H'_i} |$. Assume that $| {H'_I} | \geq \frac{1}{2} |{\bigcup_{i \in [m]} H'_i}|$, where the opposite case $|{H'_J}| \geq \frac{1}{2}\size{\bigcup_{i \in [m]} H'_i}$ is handled in the same way.

Now $F^{(m)}$ and $H'_I$ correspond to a bi-clique in $G$, and their sizes are bounded below by
\[
|{F^{(m)}}| \geq \textstyle\frac{1}{2^m} \size{\left( F_1 \right)_R} \geq \frac{1}{3 \cdot 2^{m-1}}|F_1|
\]
and
\[
|{H'_I}| \geq \textstyle\frac{1}{2}\size{\bigcup_{i \in [m]} H'_i} \geq \frac{1}{4}\size{\bigcup_{i \in [m]} H_i} \geq \frac{m}{4k} |{\overline{\left( F_2 \right)_R}}| \geq \frac{m}{12k} |F_2|.
\]

We now take $m = \floor{\log_2 k - \log_2 \ln k} \geq \frac{1}{2}\log_2 k$, which  yields $| {F^{(m)}} | \geq \frac{2}{3} \frac{\ln k}{k} |F_1|$ and $| {H'_I} | \geq \frac{1}{20} \frac{\ln k}{k} |F_2|$, which produces the desired bi-clique in $G$. 

\medskip

We now show how to construct the promised subfamilies $\{F^{(j)}\}_{j \in [k]}$ and $\{H'_j\}_{j \in [k]}$. As stated above, let $F^{(0)} = \left( F_1 \right)_R$. Fix $j \in [k]$, and assume that $F^{(i)}$ and $H'_i$'s are inductively constructed for every $i < j$.
Consider the path $P_T(v_j, s_j)$, and take the vertex $a_j$ on it which is closest to $v_j$ and satisfies that the subpath $P_T(v_j, a_j)$ contains at least half of the members of $F_j$. Note that there is a unique member $X_j \in F_j$ which is contained in $P_T(v_j, a_j)$ and has $a_j$ as an endpoint.

We distinguish two cases: either at least half of members of $F^{(j-1)}$ contain $a_j$, or at least half of them are disjoint with $a_j$. 

In the former case, let $F^{(j)}$ be those members of $F^{(j-1)}$ containing $a_j$ and let $H'_j$ consist of $X_j$ and the collection of members of $H_j$ which are not fully contained in $P_T(v_j, a_j)$. Note that the two new subfamilies $F^{(j)}$ and $H'_j$ satisfy conditions (i)-(iii) above, and that  $F^{(j)}$ and  $H'_j$ correspond to a bi-clique in $G$.

In the latter case, let $F^{(j)}$ be those members of $F^{(j-1)}$ that do not contain $a_j$ and let $H'_j$ be the collection of members of $H_j$ which are fully contained in $P_T(v_j, a_j)$. Again we note that these new subfamilies satisfy all three conditions, and that $F^{(j)}$ and $H'_j$ correspond to a bi-clique in $\overline{G}$. This finishes the inductive step of construction and concludes the proof.
\end{proof}

The asymptotically matching lower bound for the case of chordal graphs in Theorem \ref{thm:ceh} is a consequence of the following.  

\begin{theorem} \label{t:HomogeneousTreeLeaf}
Let $k \geq 17$ be an integer and let $T$ be a tree with $k$ leaves. There exist two subtree families $F_1$, $F_2$ of $T$ with the following property: If $H_1\subset F_1$ and $H_2\subset F_2$ are such that either every member of $H_1$ intersects every member of $H_2$, or every member of $H_1$ is disjoint from every member of $H_2$,
then $\size{H_i} \leq \frac{2 \ln k}{k\ln 2}  \size{F_i}$ for some $i = 1, 2$.
\end{theorem}

\begin{remark*}
To prove Theorem~\ref{t:HomogeneousTreeLeaf}, we give a construction, where $F_1$ and $F_2$ have different sizes. By duplicating vertices, we can construct $F_1'$ and $F'_2$ with equal size so that they satisfy the statement of Theorem~\ref{t:HomogeneousTreeLeaf}.
Let $F_1$ and $F_2$ be any families that satisfy Theorem~\ref{t:HomogeneousTreeLeaf}, say $|F_1|=n$ and $|F_2|=m$ for some positive integers $n$ and $m$. For any integer $t \ge 1$, we can take $F_1'$ as the multiset having $mt$ copies of each element of $F_1$ and $F_2'$ as be the multiset having $nt$ copies of each element of $F_2$.
Then $|F_1'|=|F_2'|$ and definitely $F_1'$ and $F_2'$ also satisfy Theorem~\ref{t:HomogeneousTreeLeaf}.
\end{remark*}

We split the proof of Theorem \ref{t:HomogeneousTreeLeaf} into two steps. First, we show that every bipartite graph can be ``realized" as an intersection graph between two subtree families of some tree. Then we complete the proof by showing the existence of a bipartite graph $G$ without a large bi-clique in $G$ or in $\overline{G}$, which is realized as subtree families of a tree with at most $k$ leaves.

\begin{lemma}\label{l:BipartiteAsSubtree}
Let $G \subseteq K_{m, n}$ be a bipartite graph with $2 \leq m \leq n$. There exists a tree $T$ with $m$ leaves and two subtree families $F_1$ and $F_2$ of $T$ such that the intersection graph between $F_1$ and $F_2$ is isomorphic to $G$.
\end{lemma}

\begin{proof}
Let $V = V_1 \cupdot V_2$ be the vertex partition of $K_{m, n}$ and say $V_1 = \{w_1, w_2, \ldots, w_m\}$. Consider a star $T = K_{1, m}$ on $\{u, v_1, v_2, \ldots, v_m\}$ where $u$ is the vertex of degree $m$ in $T$. For each vertex $x \in V_2$, let $G_x$ be the subtree of $T$ with edge set $E(G_x) = \{ v_i u : w_i x \in G \}$. i.e. $G_x = \tree_T(\{v_i : w_i x \in G\})$. Note that $G_x$ is isomorphic to the star of $x$ in $G$. Define the first subtree family $F_1$ as
$$
F_1 = \left\{ G_x : x \in V_2 \right\}.
$$
For each $i \in [m]$, let $H_i$ be the tree on $\{v_i\}$ with no edges. The second family consists of all such ``singletons" $H_i$:
$$
F_2 = \left\{ H_i : i \in [m] \right\}.
$$
Two trees $G_x$ and $H_i$ intersect if and only if $w_i x \in G$, showing that the intersection graph between $F_1$ and $F_2$ is isomorphic to $G$.
\end{proof}

From now on, let $k \geq 17$ be a fixed integer. For simplicity let $c = c(k) = \frac{2}{\ln 2} \frac{\ln k}{k}$. Note that $c(k) < \frac{1}{2}$ for every $k \geq 17$.

Let $n \geq k$ be an integer and consider a random graph $G \subseteq K_{k, n}$ formed by independently choosing each edge of $K_{k, n}$ with probability $\frac{1}{2}$. Let $a = \ceil{ck}$ and $b = \ceil{cn}$ be integers.
Let $X$ be the total number of copies of $K_{a,b}$ in $G$. We show $2E[X] < 1$ 
for sufficiently large $n$ so that there is some $G$ that $K_{a, b}$ is contained in neither of $G$ nor $\overline{G}$ as a subgraph. Then the subtree representation of $G$ by the subtree families $F_1$ and $F_2$ of $T = K_{1, k}$ provided by Lemma~\ref{l:BipartiteAsSubtree} satisfies Theorem~\ref{t:HomogeneousTreeLeaf}.

By linearity of expectation, we have
$$
E[X] = \binom{k}{a} \; \binom{n}{b} \; 2^{1 - ab}.
$$
In order to estimate $E[X]$, we need the following lemma for binomial coefficients.

\begin{lemma}
Let $r \in (0, 1)$ be a rational number. Let $d$ be a real number such that $\frac{1}{r^r (1-r)^{1-r}} < d$. For every sufficiently large $n$ such that $rn$ is an integer, it holds that $\binom{n}{rn} < d^n$.
\end{lemma}

\begin{proof}
By Stirling's approximation, we have $\displaystyle{\lim_{n \to \infty} \sqrt[n]{n!} = \frac{n}{e}}$. Using this formula one can easily show that $\displaystyle{\lim_{n \to \infty} \binom{n}{rn}^{1/n} = \frac{1}{r^r (1-r)^{1-r}}}$.
\end{proof}

Let $r \in (c, \frac{1}{2})$ be a rational number slightly larger than $c(k)$. For sufficiently large $n$ such that $rn$ is an integer and $rn \geq b$, we bound $E[X]$ from above as:
$$
E[X] = \binom{k}{a} \binom{n}{b} 2^{1 - ab} \leq 2^k \left( \frac{1}{r^r (1-r)^{1-r}} \right)^n 2^{1 - c^2 nk} = 2^{k+1} \left( \frac{1}{2^{c^2 k} r^r (1-r)^{1-r}} \right)^n
$$
Now our goal is to show $2^{c^2 k} r^r (1-r)^{1-r} > 1$ so that $E[X] < 1$ for sufficiently large $n$. Taking logarithm, the inequality is equivalent with:
$$
0 < c^2 k \ln 2 + r \ln r + (1-r) \ln (1-r).
$$
Putting $ck = \frac{2 \ln k}{\ln 2}$ yields
\begin{align*}
0 & < 2c \ln k + r \ln r + (1-r) \ln (1-r) = r \left(\frac{2c}{r} \ln k + \ln r \right) + (1-r) \ln (1-r)\\
& = r \ln k^{2c/r} r + (1-r) \ln (1-r).
\end{align*}
One can easily check that $r \ln \frac{1}{r} + (1-r) \ln(1-r) > 0$ for every $r \in (0, \frac{1}{2})$. Thus it is enough to show that $k^{2c/r} r \geq \frac{1}{r}$ or equivalently $k^{2c/r} r^2 \geq 1$ for our choices of $k, c$ and $r$. However, this easily follows from the continuity of an auxiliary function $f(x) = k^{2c/x} x^2$ at $x = c$ and the fact that $f(c) = k^2 c^2 = \left( \frac{2 \ln k}{\ln 2} \right)^2 > 1$.


\section{Concluding remarks} \label{sec:concl}

Recall that intersection graphs of planar convex sets have the SEH-property. On the other hand, they do not enjoy the CEH-property.
This can be easily seen by Theorem~\ref{t:HomogeneousTreeLeaf} and the following Lemma~\cite{Kim2014}.

\begin{lemma} \label{l:SubtreeNerveAsConvex}
Let $T$ be a tree and $F$ be a family of subtrees of $T$. There is a family $C$ of convex sets in $\R^2$ such that $C$ and $F$ have isomorphic nerve complexes.
\end{lemma}

It is natural to ask whether intersection graphs of convex sets in higher dimensions satisfy SEH- or CEH-properties. However, it is already pointless to consider those properties in dimension three since every graph can be realized as the intersection graph of some convex sets in $\R^3$ \cite{Tietze3d}. Another direction is to consider Erd\H{o}s-Hajnal type properties in the class of \emph{intersection hypergraphs}. We conjecture that intersection 3-uniform hypergraphs of planar convex sets satisfy the following generalization of SEH-property.

\begin{conj}\label{c:SEHof3unif}
There exists a constant $c > 0$ that makes the following hold. For every finite family $F$ of convex sets in $\R^2$ (or in $\mathbb{R}^3$), we can find pairwise disjoint subfamilies $F_1$, $F_2$, $F_3$ $\subseteq F$ of size $\size{F_i} \geq c \size{F}$ for every $i = 1, 2, 3$ such that either every rainbow triple of $F_1$, $F_2$, $F_3$ intersect or every such rainbow triple do not intersect.
\end{conj}


\bibliographystyle{plain} 

\begin{thebibliography}{10}

\bibitem{alon-semi2005}
N.~Alon, J.~Pach, R.~Pinchasi, R.~Radoi{\v c}i{\' c}, M.~Sharir.
Crossing patterns of semi-algebraic sets.
Journal of Combinatorial Theory, Series A,
Volume 111 (2005), 310--326.

\bibitem{ChudnovskySurvey}
M.~Chudnovsky,
The Erd\H{o}s-Hajnal conjecture---a survey.
J. Graph Theory 75 (2014), 178--190.

\bibitem{Erdos1947}
P.~Erd\H{o}s,
Some remarks on the theory of graphs.
Bull. Amer. Math. Soc. 53 (1947), 292--294.

\bibitem{ErdosHajnal1989} P.~Erd\H{o}s~and~A.~Hajnal,
Ramsey-type theorems.
Discrete Appl. Math. 25 (1989), 37--52.

\bibitem{Fox2006} J.~Fox,
A bipartite analogue of Dilworth's theorem.
Order 23 (2006), 197--209.

\bibitem{FoxGromovPach2012} J.~Fox,~M.~Gromov,~V.~Lafforgue,~A.~Naor~and~J.~Pach,
Overlap properties of geometric expanders.
J. Reine Angew. Math. 671 (2012), 49--83.

\bibitem{FoxPach2008}
J.~Fox~and~J.~Pach,
Erd\H{o}s-Hajnal-type results on intersection patterns of geometric objects.
in~\textit{Horizons~of~combinatorics},~vol.~17~of~\textit{Bolyai~Soc.~Math.~Stud.},~79--103,~Springer,~Berlin,~2008.

\bibitem{FoxPachToth2010}
J.~Fox, J.~Pach, C.~T{\'o}th.
Tur{\'a}n-type results for partial orders and intersection graphs of convex sets.
Israel J. Math 178 (2010), 29--50.

\bibitem{Gavril1974} F.~Gavril,
The intersection graphs of subtrees in trees are exactly the chordal graphs.
J. Combinatorial Theory Ser. B 16 (1974), 47--56.

\bibitem{Kim2014} M.~Kim,
Intersection patterns of subtree families and colorful fractional Helly theorems.
Master's thesis, Korea Adv. Inst. Science. Techn., Daejeon, Republic of Korea, 2014.

\bibitem{LarmanMatousekPachTorocsik}
D.~Larman,~J.~Matou\v{s}ek,~J.~Pach~and~J.~T\"{o}r\H{o}csik,
A Ramsey-type result for convex sets.
Bull. London Math. Soc. 26 (1994), 132–-136.

\bibitem{LekkeikerkerBoland}
C.~G.~Lekkerkerker~and~J.~Ch.~Boland,
Representation of a finite graph by a set of intervals on the real line.
Fund. Math 51 (1962/63), 45--64.

\bibitem{PachSegments2001}
J.~Pach, J.~Solymosi.
Crossing patterns of segments.
J. Combin. Theory Ser. A 96 (2001), 316--325.

\bibitem{PachSolymosiProc}
J.~Pach, J.~Solymosi.
Structure theorems for systems of segments.
In: J.~Akiyama,  M.~Kano, M.~Urabe (eds), JCDCG 2000. Lecture Notes in Computer Science, vol 2098. 

\bibitem{Tietze3d} H.~Tietze,
\"Uber das Problem der Nachbargebiete im Raum.
Monatsh. Math. Phys. 16 (1905), 211--216.

\end{thebibliography}

\end{document}